\documentclass[a4paper]{amsart}

\usepackage{amsmath,amssymb,amsthm}
\usepackage{hyperref}
\hypersetup{%
  colorlinks,
  bookmarksopen,
  bookmarksnumbered,
  citecolor=blue,
  linkcolor=black,
  pdfstartview=FitH,
}

\usepackage{color}
\newtheorem{lem}{Lemma}[section]
\newtheorem{eg}[lem]{Example}
\newtheorem{prop}[lem]{Proposition}
\newtheorem{cor}[lem]{Corollary}
\newtheorem{thm}[lem]{Theorem}

\theoremstyle{definition}
\newtheorem{ques}[lem]{Question}
\newtheorem{quesc}[lem]{\llap{*}Question}
\theoremstyle{remark}
\newtheorem{rk}[lem]{Remark}
\newtheorem{rkc}[lem]{\llap{*}Remark}
\newenvironment{proofc}{\begin{proof}[\llap{*}Proof]}{\end{proof}}

\let\Re\relax\DeclareMathOperator{\Re}{Re}
\let\Im\relax\DeclareMathOperator{\Im}{Im}
\DeclareMathOperator{\spn}{span}
\DeclareMathOperator{\tgm}{tgm}

\DeclareMathOperator{\conv}{conv}
\DeclareMathOperator{\trace}{trace}

\let\epsilon\varepsilon\let\phi\varphi

\newenvironment{smallbmatrix}{\big[\begin{smallmatrix}}{\end{smallmatrix}\big]}
\newenvironment{smallvect}[1][\Bigg]%
{\let\privatefoo#1\privatefoo[\begin{smallmatrix}}%
  {\end{smallmatrix}\privatefoo]}
\newcommand{\Wme}{W_{\!\mathrm{m},\mathrm{e}}}
\newcommand{\Wm}{W_{\!\mathrm{m}}}
\newcommand{\svdots}{\\[-.8ex]\vdots\\[.4ex]}
\let\sectmark\S
\def\dc#1{\expandafter\def\csname#1\endcsname{\mathcal{#1}}}
\def\db#1{\expandafter\def\csname b#1\endcsname{\mathbb{#1}}}

\def\loopy#1#2{%
  \def#1##1{\def\next{#2{##1}#1}\ifx##1\relax\let\next\relax\fi\next}}

\loopy{\makemathcals}{\dc}\loopy{\makemathbbs}{\db}

\makemathbbs  CRNQZT\relax
\makemathcals QWERTYUIOPASDFGHJKLZXCVBNM\relax

\makeatletter\newcount\@DT@modctr\newcount\@dtctr
\def\@modulo#1#2{\@DT@modctr=#1\relax\divide \@DT@modctr by #2\relax
\multiply \@DT@modctr by #2\relax\advance #1 by -\@DT@modctr}
\newcommand{\xxivtime}{\@dtctr=\time\divide\@dtctr by 60
\ifnum\@dtctr<10 0\fi\the\@dtctr:\@dtctr=\time\@modulo{\@dtctr}{60}%
\ifnum\@dtctr<10 0\fi\the\@dtctr}

\makeatother

\newcommand{\hsnorm}[1]{%
  \left\vert\kern-0.9pt\left\vert\kern-0.9pt\left\vert #1
      \right\vert\kern-0.9pt\right\vert\kern-0.9pt\right\vert}

\begin{document}

\title[CB norms of right module maps]{Completely bounded norms\\ of right module maps}
\author{Rupert H. Levene}
\email{rupert.levene@ucd.ie}
\address{School of Mathematics and Statistics\\University College Dublin\\Dublin 4\\Ireland}
\author{Richard M. Timoney}
\email{richardt@maths.tcd.ie}
\address{School of Mathematics\\Trinity College Dublin\\Dublin 2\\Ireland}
\maketitle

\begin{abstract}
  It is well-known that if~$T$ is a $D_m$--$D_n$ bimodule map on the
  $m\times n$ complex matrices, then~$T$ is a Schur multiplier and
  $\|T\|_{cb}=\|T\|$.  If~$n=2$ and~$T$ is merely assumed to be a
  right $D_2$-module map, then we show that
  $\|T\|_{cb}=\|T\|$. However, this property fails if $m\geq 2$ and
  $n\geq 3$. For $m\ge2$ and $n=3,4$ or $n\ge m^2$ we give examples of
  maps~$T$ attaining the supremum \[C(m,n)=\sup\{\|T\|_{cb}\colon
  \text{$T$ a right $D_{n}$-module map on~$M_{m,n}$ with
    $\|T\|\leq1$}\},\] we show that $C(m,m^2) = \sqrt{m}$ and succeed
  in finding sharp results for $C(m,n)$ in certain other cases.  As a
  consequence, if~$H$ is an infinite-dimensional Hilbert space and $D$
  is a masa in $B(H)$, then there is a bounded right $D$-module map on
  $\K(H)$ which is not completely bounded.
  \\[\smallskipamount]
  \\[\smallskipamount]
\end{abstract}

\section{Introduction}

Let~$H$ be a Hilbert space, let~$\B(H)$ be the algebra of bounded
linear operators on~$H$, let~$\K(H)$ be the ideal of compact operators
and let~$D$ be a masa in~$\B(H)$. If~$T\colon\K(H)\to \K(H)$ is a
bounded $D$-bimodule map, then it is well-known that that
$\|T\|_{cb}=\|T\|$ (see~\cite{smith91,paulsen,PSS}).
While it would certainly be of use to be able to extend this to larger
natural classes than $D$-bimodule maps (generalised Schur multipliers),
in the present
paper, we consider the effect of relaxing the hypothesis of
bimodularity to one-sided modularity over~$D$.
While we establish a positive result for dimension $2$, we give
increasing bounds for higher finite dimensions 
and a negative answer for the following
question~\cite[Remark~7.10]{JLTT}:
\begin{ques}\label{ques:JLTT}
  If~$H$ is infinite-dimensional and $D$ is a masa in $\B(H)$, is
  there a constant $C>0$ such that $\|T\|_{cb}\leq C\|T\|$ for every
  bounded, left $D$-module map~$T\colon \K(H)\to \K(H)$?
\end{ques}
By symmetry, this question is unchanged if we replace ``left'' by
``right'', and this makes our notation marginally neater.  So we will
focus on right $D$-module maps.

Of course, if~$H$ is finite dimensional, then the answer to this
question is yes even if we discard the modularity condition. It then
becomes interesting to estimate the optimal constant~$C$. Hence we are
led to consider the constants
\[C(m,n)=\sup\{\|T\|_{cb} \colon \text{$T$ is a right $D_n$-module map
  on~$M_{m,n}$},\ \|T\|\leq 1\}\] where $M_{m,n}$ is the space of
$m\times n$ complex matrices and $D_n$ is the algebra of diagonal
$n\times n$ matrices.

The structure of the paper is as follows. We first establish some
notation and give some preliminary results in
Section~\ref{sec:prelim}. In Section~\ref{sec:twocol} we use the
second author's work on elementary operators to show that $C(m,2)=1$
for every $m\ge1$. Section~\ref{sec:CBnormsHS} contains some technical
results comparing the completely bounded norm to the norm arising from
the Hilbert-Schmidt norm, and these are used in
Section~\ref{sec:upperbounds} to find some upper bounds for
$C(m,n)$. In the next section we construct examples which show
that~$C(m,n)$ grows with $m,n$. This leads naturally to a
counterexample (in Corollary~\ref{cor:answer}) answering
Question~\ref{ques:JLTT}, and we are also able to determine the values
of $C(m,n)$ in some cases. Finally, in Section~\ref{sec:failures} we
briefly consider similar problems when we restrict attention to
special classes of right module maps.

In the last two sections, we
pose several unresolved questions about the behaviour of the
constants~$C(m,n)$.
\smallskip

This version of the paper incorporates several corrections to~\cite{lt-laa}, as
detailed in~\cite{lt-corr}. Items with corrections are marked with an asterisk.

\section{Preliminaries}\label{sec:prelim}

If~$X$ is a vector space, we write $\L(X)$ for the space of linear
maps $X\to X$. If $m,n\in\bN$, then $M_{m, n}(X)$ is the vector space
of~$m\times n$ matrices with entries in~$X$. We will write elements of
$M_{m, n}(X)$ as $[x_{ij}]_{1\leq i\leq m,\,1\leq j\leq n}$ or simply
$[x_{ij}]$, where each $x_{ij}$ is in~$X$. If~$T\in \L(X)$ and $m,n\in
\bN$, then the $(m,n)$-ampliation of~$T$ is the map $T_{m,n}\in
\L(M_{m, n}(X))$ given by $T_{m,n}[x_{ij}]=[Tx_{ij}]$. We also write
$T_n=T_{n,n}$.

Given a norm~$\|\cdot\|$ on~$X$, the corresponding operator norm, or
simply the norm, of a map $T\in \L(X)$ is
\[ \|T\|=\sup\{\|Tx\|\colon x\in X,\ \|x\|\leq 1\}.\] If we are given
norms on $M_{m,n}(X)$ for all $m,n\in \bN$, then the completely bounded
norm of~$T$ is
\[\|T\|_{cb}=\sup_{m,n\ge1}\|T_{m,n}\|.\] Provided the inclusions of
$M_{m,n}(X)$ into $M_{m+1,n}(X)$ and $M_{m,n+1}(X)$ which pad a matrix
with an extra row or column of zeros are isometries, we have
\[\|T\|=\|T_1\|\leq\|T_2\|\leq\|T_3\|\leq\dots\leq
\|T\|_{cb}=\sup_{n\ge1}\|T_n\|.\]

For $n\in \bN$ we let~$\bC^n$ denote the Hilbert space of
dimension~$n$ whose elements are to be thought of as column vectors
with $n$ complex entries, with the $\ell^2$ norm, and we will also
write $\bC^\infty$ for $\ell^2(\bN)$. For $m,n\in \bN\cup\{\infty\}$,
we write \[M_{m, n}=\B(\bC^n,\bC^m)=\{ x\in\L(\bC^n,\bC^m) \colon
\|x\|<\infty\}\] and $M_m=M_{m, m}$. If~$s,t\in\bN$, then
$M_{s,t}(M_{m, n})$ can be naturally identified with the normed vector
space~$M_{sm, tn}$, and hence inherits the norm from the latter
space. Adding a row or column of zeros is then an isometry.

If~$v,w\in \bC^n$, then $vw^*$ denotes the rank one operator in~$M_n$
given by \[vw^*(x)=\langle x,w\rangle v\quad\text{for $x\in \bC^n$.}\]
For~$1\leq i\leq n$ (or for $i\ge1$, if $n=\infty$) we write $e_i$ for
the $i$th standard basis vector in $\bC^n$. Then $D_n$, the diagonal
masa of~$M_n$, is the von Neumann algebra generated by the diagonal
matrix units~$e_ie_i^*$.

Let~$n\in\bN\cup \{\infty\}$, let~$b_1,\dots,b_\ell\in M_n$ and let
$b=\begin{smallvect} b_1\svdots b_\ell
\end{smallvect}$.  For $\xi\in \bC^n$, let $Q(b,\xi)$ be the
positive semi-definite $\ell\times \ell$
matrix \[Q(b,\xi)=[\langle b_i\xi,b_j\xi\rangle]_{1\leq i,j\leq
  \ell}.\] We recall the definitions from~\cite{timoney03} of
\textit{the matrix numerical range of\/ $b$},
\[\Wm(b)= \{ Q(b,\xi) \colon \xi\in\bC^n,\ \|\xi\|=1\}\] and
\emph{the matrix extremal numerical range of\/~$b$},
\[ \Wme(b)=\{ \beta \in \overline{\Wm(b)} \colon
\trace(\beta)=\|b\|^2\},\] (where the norm $\|b\|$ is computed with
respect to the norm on $M_{\ell,1}(M_n)$ described above). It is easy
to see that $\Wme(b)$ is the set of elements of the closure
of~$\Wm(b)$ of maximal trace. If~$n<\infty$ then $\Wm(b)$ is a
continuous image of the unit sphere of~$\bC ^n$, which is
compact. Hence in this case,
\[\Wme(b)=\{ Q(b,\xi) \colon \xi\in \bC^n,\ \|\xi\|=1,\
b^*b\xi=\|b\|^2\xi\}.\] Observe that the vectors $\xi$ appearing in
this expression are precisely the unit vectors in the eigenspace
of~$b^*b$ corresponding to its maximal eigenvalue.

If $a=[a_1\ \dots\ a_\ell]$ and $b=
\begin{smallvect}
  b_1\svdots b_\ell
\end{smallvect}$ for some $a_j\in \L(X)$ and $b_j\in \L(Y)$, then we
will write $T=a\odot b$ or say that ``$a,b$ represent~$T$'' to mean
that $T$ is the elementary operator \[T\colon\L(Y,X)\to \L(Y,X),\quad
x\mapsto \sum_{j=1}^\ell a_j xb_j.\] Such a representation of~$T$ is
far from unique due to bilinearity in $(a, b)$;
for example, if $T=a\odot b$, then we also have
$T=(a\alpha^{-1})\odot(\alpha b)$ for any invertible matrix $\alpha\in
M_\ell$.

If~$D$ is a subring of~$M_n$ then $M_{m,n}$ is a right
$D$-module. A \emph{right $D$-module map} on~$M_{m,n}$ is a linear map
$T\in\L(M_{m,n})$ such that \[T(xd)=T(x)d\quad\text{for all $x\in M_{m,n}$ and
all $d\in D$}.\] We write $\L_{D}(M_{m,n})$ for the set of all right
$D$-module maps on~$M_{m,n}$.\medskip
\goodbreak

\begin{rk}\label{rk:rep}
  If $n\in\bN$ and $T$ is a bounded right $D_n$-module map on~$M_{m,
    n}$, then~$T$ is an elementary operator of the
  form~$Tx=\sum_{j=1}^n a_jxb_j$ for some $b_j\in D_n$ and $a_j\in
  M_m$.  Indeed, for each~$j$, the map $v\mapsto T(ve_j^*)e_j$ is
  linear $\bC^m\to\bC^m$, and it is bounded since~$T$ is
  bounded. Hence there is an operator~$a_j\in M_m$ such that $a_j
  v=T(ve_j^*)e_j$ for $v\in\bC^m$.  We call the operators $a_j$ the
  \emph{column operators} of~$T$.  Writing~$b_j=e_je_j^*$, we have
  \[ \sum_{j=1}^n a_jxb_j = \sum_{j=1}^n a_jxe_je_j^* =\sum_{j=1}^n
  T(xe_je_j^*)e_je_j^* =\sum_{j=1}^n T(x)e_je_j^*= T(x).\] We have
  found a representation $T=a'\odot b'$ where $a'=[a_1\ \dots a_n]$
  and $b'=\begin{smallvect}
  b_1\svdots b_n
\end{smallvect}$, and each $b_j$ is diagonal. As
  discussed in~\cite[\sectmark 3]{timoney03}, there is a
  representation~$T=a\odot b$ where the entries of $a$ and $b$ are
  linear combinations of the entries of $a'$ and $b'$, respectively,
  so that
  \[\|T\|_{cb}=\|a\|\,\|b\|=\tfrac12(\|a\|^2+\|b\|^2).\] Observe that
  the entries of $b$ are then diagonal, and $\|a\|=\|b\|$ by the
  arithmetic mean/geometric mean
  inequality. In~\cite[Theorem~3.3]{timoney03}, the second author
  shows that a representation $T=a\odot b$ satisfies these equalities
  if and only if
  \begin{equation}\tag{$\bigstar$}\label{eq:Wme}
    \conv \Wme(a^*)\cap \conv \Wme(b)\ne \emptyset
  \end{equation}
  where $\conv S$ denotes the convex hull of a subset~$S$ of a
  vector~space.
  
  If~$n=\infty$, so that~$T$ is a bounded right $D_\infty$-module map
  on $\B(H,\bC^m)$ where $H=\ell^2(\bN)$, then the same argument gives
  $Tx=\sum_{j=1}^\infty a_jxb_j$ where the operators $a_j\in M_m$ are
  given by $a_jv=T(ve_j^*)e_j$ and $b_j=e_je_j^*\in \B(H)$, and the
  series converges in the strong operator topology.
\end{rk}

The relevance of the following lemma to our problem is plain in light
of Remark~\ref{rk:rep}, and condition~(\ref{eq:Wme}) in particular.
\begin{lem}\label{lem:conv}
  Let~$n\in\bN$,
  let $\ell\in\bN$ and let $b_1,\dots,b_\ell\in D_n$.  If~$b=
  \begin{smallvect}
    b_1\svdots b_\ell
  \end{smallvect}$, then
  \[\Wme(b)=\conv\{Q(b,e_p) \colon 1\leq p\leq n,\ b^*be_p=\|b\|^2e_p\}.\]
  In particular, $\Wme(b)$ is convex.
\end{lem}
\begin{proof}
  The matrix $b^*b=\sum_{j=1}^\ell b_j^*b_j$ is positive semi-definite
  and diagonal with largest eigenvalue~$\|b\|^2$. Let~$r$ be the
  dimension of the corresponding eigenspace.  Permuting
  $e_1,\dots,e_n$ if necessary, we have $b^*b=\|b\|^2(I_r \oplus d)$
  for some positive semi-definite $d\in D_{n-r}$ with $\|d\|<1$. So
  $Q(b,\xi)\in \Wme(b)$ if and only if $\xi=\sum_{p=1}^r \xi_p e_p$
  for some $\xi_p\in\bC$ such that $\sum_{p=1}^r|\xi_p|^2=1$. Each
  $b_j$ is diagonal so the vectors $e_q$ are eigenvectors, hence
  \begin{align*} Q(b,\xi)&=\Big[\sum_{p,q=1}^r \langle b_i\xi_pe_p,b_j \xi_q
  e_q\rangle\Big] \\&= \Big[\sum_{p=1}^r \langle
  b_i\xi_pe_p,b_j\xi_pe_p\rangle\Big]\\&= \sum_{p=1}^r |\xi_p|^2
  Q(b,e_p).\qedhere
\end{align*}
\end{proof}
\goodbreak

The following argument is essentially contained in any
of~\cite{paulsen,PSS,smith91}.
\begin{lem}\label{lem:smith}
  Let~$H,K$ be Hilbert spaces, let~$X$ be a subspace of~$\B(H,K)$, and
  let $A\subseteq \B(H)$ be a right norming set for~$X$, meaning
  that~$xa\in X$ for all $x\in X$ and $a\in A$, and for every~$n\ge1$
  and every $z\in M_n(X)$, we have
  \[ \|z\|_{M_n(X)} = \sup \{ \|zb\|_{M_{n,1}(X)} \colon b\in
  M_{n, 1}(A),\ \|b\| \leq 1\}.
  \]
  If $T\colon X\to X$ is a bounded, linear map such that $T(xa)=T(x)a$
  for all $x\in X$, $a\in A$, then $\|T_n\|=\|T_{n,1}\|$ for all
  $n\ge1$.
\end{lem}
\begin{proof}
  The inequality $\|T_{n,1}\|\leq \|T_n\|$ is
  clear. On the other hand, if~$z\in M_n(X)$ and $b\in M_{n, 1}(A)$,
  then 
  \[T_n(z)b=\Big[\sum_j T(z_{ij})b_j\Big]_i = \Big[\sum_j
  T(z_{ij}b_j)\Big]_i = T_{n,1}(zb)\] and $\|zb\| \leq
  \|z\|_{M_n(X)}\,\|b\|_{M_{n, 1}(A)}$. Since $A$ is a right norming
  set for~$X$, 
  \begin{align*}
    \|T_n\| &=\sup\{ \|T_n(z)\|_{M_n(X)}\colon z\in M_n(X),\ \|z\|\leq 1\}\\
    &= \sup\{\|T_{n,1}(zb)\|_{M_{n,1}(X)} \colon b\in M_{n, 1}(A),\ \|b\|\leq 1,\
    z\in M_n(X),\ \|z\|\leq 1\}\\& \leq \|T_{n,1}\|.\qedhere
  \end{align*}
\end{proof}

As shown in~\cite{paulsen,smith91}, the set $D_n$ of diagonal matrices
in~$M_n$ is a right norming set for~$M_{m, n}$.  Thus we immediately
obtain:
\begin{prop}\label{prop:col}
  If~$m,n\in \bN\cup\{\infty\}$ and $T$ is a right $D_n$-module map on
  $M_{m, n}$, then $\|T\|_{cb}=\sup_{k\ge1}\|T_{k,1}\|$.\qed
\end{prop}

\begin{rk}\label{rk:one}
  If $m=1$ or $n=1$ (that is, if the matrices on which our maps act
  have either one row, or one column) then $\|T\|_{cb}=\|T\|$ for
  every $T\in \L(M_{m,n})$. For if $n=1$, then $M_{m, n}=\bC^m$, and
  every linear map $T\colon \bC^m\to \bC^m$ may be written as $Tx =
  ax$ for $a \in M_m$. Hence $\|T\| = \|a\|$. Moreover $T_k \colon
  M_k(\bC^m) \to M_k(\bC^m)$ is given by left multiplication by a
  block diagonal matrix $a^{(k)}$ (with $k$ copies of $a$ on the
  diagonal), and $\|T_k\| = \| a^{(k)}\| = \|a\| =\|T\|$, so
  $\|T\|_{cb}=\|T\|$.  If $m=1$, we can apply a similar argument with
  right multiplication or use the previous case on the map $T^*\colon
  M_{n,m}\to M_{n,m}$ given by $T^*(x)= T(x^*)^*$.
\end{rk}

\section{Two columns}\label{sec:twocol}

We now show that, surprisingly, the conclusion $\|T\|_{cb}=\|T\|$ of
Remark~\ref{rk:one} persists for right $D_2$-module maps on~$M_{m,2}$.

\begin{lem}\label{lem:convproj}
  If~$X$ is a set of positive semi-definite $2\times 2$ matrices with
  trace\/~$1$ and there is a rank one projection $p\in \conv X$, then
  $p\in X$.
\end{lem}
\begin{proof}
  Conjugating by a suitable unitary matrix, we may assume that
  $p=e_1e_1^*$. Now~$p$ is a convex combination of some
  $\alpha_1,\dots,\alpha_k\in X$ and each $\alpha_j$ is positive
  semi-definite. Since the $(2,2)$ entry of $p$ is zero, the $(2,2)$
  entry of each $\alpha_j$ is~zero, which implies that the
  off-diagonal entries of each $\alpha_j$ are also zero. Since
  $\trace\alpha_j=1$, we have $\alpha_j=p$ for all~$j$.
\end{proof}

\begin{thm}\label{thm:2col}
  If~$m\in\bN$ and $T\colon M_{m, 2}\to M_{m, 2}$ is a right
  $D_2$-module map, then $\|T\|_{cb}=\|T\|$.
\end{thm}
\begin{proof}
  Suppose $\|T\|_{cb}=1$.  By Remark~\ref{rk:rep},
  $Tx=a_1xb_1+a_2xb_2$ for some $a_1,a_2\in M_m$ and $b_1,b_2\in D_2$
  such that $\|a\|=\|b\|=1$ where $a=[a_1\ a_2]$ and $b=
  \begin{smallbmatrix}
    b_1\\b_2
  \end{smallbmatrix}$.  By Lemma~\ref{lem:conv}, $\Wme(b)$ is convex,
  so by~\cite[Theorem~3.3]{timoney03}, $\Wme(b)$ intersects the convex
  hull of $\Wme(a^*)$. By~\cite[Proposition~3.1]{timoney03}, it
  suffices to show that $\Wme(a^*)\cap \Wme(b)\ne \emptyset$.

  Observe that $b^*b=b_1^*b_1+b_2^*b_2$ is a $2\times 2$ diagonal
  positive semi-definite matrix of norm~$1$, so its
  $1$-eigenspace~$E_1(b^*b)$ has dimension~$1$ or~$2$.

  If $\dim E_1(b^*b)=1$, then $b^*b=
  \begin{smallbmatrix}
    1&0\\0&t
  \end{smallbmatrix}$ or $b^*b=
  \begin{smallbmatrix}
    t&0\\0&1
  \end{smallbmatrix}$ for some $t\in [0,1)$. 
  If $b^*b= \begin{smallbmatrix} 1&0\\0&t
  \end{smallbmatrix}$, then \[\Wme(b)=\{ Q(b, ze_1)\colon z\in\bT\} =
  \{ Q(b,e_1)\}=\{e_1e_1^*\}.\] Since $e_1e_1^*$ is a rank one
  projection in $\conv \Wme(a^*)$, we have $e_1e_1^*\in \Wme(a^*)$ by
  Lemma~\ref{lem:convproj}. Hence $\Wme(a^*)\cap \Wme(b)\ne\emptyset$.
  Similarly, if $b^*b=\begin{smallbmatrix} t&0\\0&1
  \end{smallbmatrix}$ then $e_2e_2^*\in \Wme(a^*)\cap
  \Wme(b)\ne\emptyset$.

  Now suppose that $\dim E_1(b^*b)=2$. Then $b^*b=I_2$, so if we write
  $\beta_i=Q(b,e_i)$ for $i=1,2$, then Lemma~\ref{lem:conv} shows that
  $\Wme(b)=\conv\{ \beta_1,\beta_2\}$. For~$i=1,2$, let us write $b_i=
  \begin{smallbmatrix}
    b_{i1}&0\\0&b_{i2}
  \end{smallbmatrix}$ and let $v_i=
  \begin{smallbmatrix}
    b_{1i}\\b_{2i}
  \end{smallbmatrix}$. A simple calculation reveals that $\|v_i\|=1$
  and $\beta_i=v_iv_i^*$. If $\beta_1=\beta_2$, then this rank one
  projection is in $\Wme(a^*)$ by Lemma~\ref{lem:convproj}, so
  $\Wme(a^*)\cap \Wme(b)\ne \emptyset$. So we may assume that
  $\beta_1\ne\beta_2$, so that $\Wme(b)$ is the proper closed line
  segment joining $\beta_1$ and $\beta_2$.

  For $t\in\bR$, let $\beta(t)=t\beta_1+(1-t)\beta_2$ and consider the
  closed convex set \[S=\{ t\in \bR\colon \beta(t)\in
  \conv\Wme(a^*)\}.\] Now $\beta_1$ and $\beta_2$ are distinct and
  $\|\beta_i\|_2=1$, where $\|\cdot\|_2$ is the Hilbert-Schmidt norm
  on~$M_2$. Moreover, $(M_2,\|\cdot\|_2)$ is strictly convex, and its
  closed unit ball contains $\Wme(a^*)$ since the trace-class norm of
  every matrix in~$\Wme(a^*)$ is~$1$, which dominates its
  Hilbert-Schmidt norm. Hence $S\subseteq [0,1]$, say $S=[s_1,s_2]$
  where $0\leq s_1\leq s_2\leq 1$, and $\beta(s_1)$ and $\beta(s_2)$
  are in the boundary of $\conv \Wme(a^*)$, and are the extreme points
  of $\conv\Wme(a^*)\cap \Wme(b)$.

  Given a hermitian $2\times 2$ matrix~$\alpha$ with trace~$1$, say
  $\alpha=\left[
  \begin{smallmatrix}
    a&b\\[.4ex]\overline{b}&1-a
  \end{smallmatrix}\right]$, let us write
  \[\theta(\alpha)=(a,\Re b, \Im b)\in \bR^3.\]
  Observe that the map $\theta$ defined on this convex set of matrices
  is injective and respects convex combinations.  Consider
  \[e=\theta(\beta(s_1)),\quad L=\theta(\Wme(b)),\quad
  W=\theta(\Wme(a^*)),\quad C=\conv W.\] By construction, $e$ is an
  extreme point of $C\cap L$ which lies in the boundary of~$C$.
  Let~$\Pi$ be a supporting hyperplane for~$C$ through~$e$, so that
  \[e\in \Pi=\{ x\in \bR^3\colon \langle x,\eta\rangle=r\}\] for some
  non-zero vector $\eta=(\eta_1,\eta_2,\eta_3)\in\bR^3$ and some
  $r\in\bR$, chosen so that \[C \subseteq \Pi^+=\{ x\in \bR^3 \colon
  \langle x,\eta\rangle \geq r\}.\]

  Since $e\in C=\conv W$ and $e\in \Pi$ we have $e\in \conv(\Pi\cap
  W)$; for otherwise, $e$ would be a proper convex combination of
  points in~$W$ involving at least one $x\in W$ with $\langle
  x,\eta\rangle >r$, hence $\langle e,\eta\rangle>r$ so $e\not\in
  \Pi$, a contradiction.

  We have
  \[W= \{ \theta(Q(a^*,\xi)) \colon \xi\in E_1(aa^*),\ \|\xi\|=1\}.\]
  Since $e\in \conv (\Pi\cap W)$ and $\Pi$ is an affine
  $2$-dimensional space, Carath\'eodory's theorem~\cite{gruber} shows
  that $e\in \conv\{w_1,w_2,w_3\}$ for some $w_1,w_2,w_3\in \Pi\cap
  W$. Choose unit vectors $\xi_1,\xi_2,\xi_3$ in $E_1(aa^*)$ so that
  $w_j=\theta(Q(a^*,\xi_j))$. Let~$F=\spn\{\xi_1,\xi_2,\xi_3\}$ and
  let
  \[ W'=\{ \theta(Q(a^*,\xi)) \colon \xi\in F,\ \|\xi\|=1\}.\] By
  construction, $e\in \conv W'$. We now wish to show that $W'$ is
  convex. Let~$p$ be the orthogonal projection $\bC^m\to F$ and
  consider the three self-adjoint operators $h_1,h_2,h_3\in \B(F)$ given by
  \[ h_1=pa_1a_1^*|_F,\quad h_2=p\Re(a_2a_1^*)|_F,\quad
  h_3=p\Im(a_2a_1^*)|_F.\]  Observe that
  \[W' = \{ ( \langle h_1 \xi,\xi\rangle,\langle
  h_2\xi,\xi\rangle,\langle h_3\xi,\xi\rangle) \colon \xi\in F,\
  \|\xi\|=1\}
  \]
  is the joint numerical range of $h_j$, $j=1,2,3$.
  Moreover, $W'\subseteq C\subseteq\Pi^+$, so if we write
  \[h=\eta_1 h_1 + \eta_2h_2+\eta_3h_3-rI_F\in \B(F), \] then $h\ge0$
  and $h\xi_j=0$ for $j=1,2,3$, so $h=0$.  Choose $j\in \{1,2,3\}$
  with $\eta_j\ne0$. Since $h=0$, the set~$W'$ is affinely equivalent
  to the joint numerical range of the pair of hermitian operators
  $\{\eta_kh_k \colon k\ne j\}$, which is convex by the
  Toeplitz--Hausdorff theorem~\cite{gust-rao}. Hence $W'$ is convex,
  so $e\in W'$ and \[\beta(s_1)=\theta^{-1}(e)\in \theta^{-1}(W')\cap
  \Wme(b) \subseteq \Wme(a^*)\cap\Wme(b)\ne\emptyset.\qedhere\]
 \end{proof}
\goodbreak

The case $m=\infty$ is now more or less immediate.
\begin{cor}\label{cor:infty2}
  If $T\colon M_{\infty,2}\to M_{\infty,2}$ is a right $D_2$-module
  map, then $\|T\|=\|T\|_{cb}$.
\end{cor}
\begin{proof}
  Otherwise, there is a counterexample~$T$ with $1=\|T\|<\|T\|_{cb}$.
  Recall that $M_{\infty, 2}=\B(\bC^2,H)$ where $H=\ell^2(\bN)$. By
  Proposition~\ref{prop:col}, there is some $k>1$ and some $x\in
  M_{k,1}(M_{\infty, 2})=\B(\bC^2,H\otimes \bC^k)$ with
  $\|x\|<1<\|T_{k,1}(x)\|$. Given $m\in \bN$, let $p_m\in \B(H)$ be
  the orthogonal projection onto the linear span of~$\{e_i\colon 1\leq
  i\leq m\}$ and consider $q_m=p_m\otimes I_k$.  Every operator in
  $\B(\bC^2,H\otimes \bC^k)$ has rank at most~$2$, so is compact, and
  $T_{k,1}$ is bounded (in fact, $\|T_{k,1}\|\leq k$). Hence there is
  $m\in\bN$ such that $\|q_mT_{k,1}(q_mx)\|>1$.  Let us identify
  $M_{m,2}$ with the subspace $p_m(M_{\infty,2})$ of~$M_{\infty,2}$,
  and consider $S\colon M_{m, 2}\to M_{m, 2}$, $y\mapsto
  p_mT(y)$. This is a right $D_2$-module map and
  \[ \|S\|\leq \|T\|=1<\|q_mT_{k,1}(q_mx)\|=\|S_{k,1}(q_mx)\| \leq
  \|S\|_{cb},\] contradicting Theorem~\ref{thm:2col}.
\end{proof}

\section{CB norms and Hilbert-Schmidt norms}\label{sec:CBnormsHS}

Given $n,m$, let $\L(M_{m, n})$ be the set of linear
maps~$M_{m, n}\to M_{m, n}$. For a map $T\in \L(M_{m, n})$,
we continue to write
\[ \|T\|=\sup\{ \|Tx\|\colon x\in M_{m, n},\ \|x\|=1\}\] for the
operator norm of~$T$ with respect to the operator norm $\|\cdot\|$
on~$M_{m, n}$, and we will also consider the quantity
\[\hsnorm{T}=\sup\{ \|Tx\|_2\colon x\in M_{m,
  n},\ \|x\|_2 =1\},\] that is, the operator norm of~$T$ with respect
to the Hilbert-Schmidt norm $\|\cdot\|_2$ on~$M_{m, n}$.  Note
that if~$n=\infty$ or $m=\infty$, then all of these ``norms'' may take
the value~$\infty$.

For $T\in \L(M_{m, n})$, let $T^*\in \L(M_{n, m})$ be the map given by
\[T^*(x)=T(x^*)^*,\quad x\in M_{n,m}.\] Clearly, $\|T^*\|=\|T\|$ and
$\hsnorm{T^*}=\hsnorm{T}$.

\begin{rk}\label{rk:hsnorm}
  The norm $\hsnorm{\cdot}$ behaves particularly nicely when we take
  ampliations: if~$T\in \L(M_{m, n})$ and $s,t\in \bN$, then viewing
  $T_{s,t}$ as a map on $M_{ms,nt}$, we have
  $\hsnorm{T_{s,t}}=\hsnorm{T}$. Indeed, the inequality
  $\hsnorm{T}\leq \hsnorm{T_{s,t}}$ is trivial, and
  \begin{align*}
    \hsnorm{T_{s,t}}=\sup \|[Tx_{ij}]\|_2=
    \sup\sqrt{\sum_{\substack{1\leq i\leq s\\1\leq j\leq t}}
      \|Tx_{ij}\|_2^2}\leq \hsnorm{T},
  \end{align*}
  where the suprema are taken over those $x_{ij}\in M_{m,n}$ (for
  $1\leq i\leq s$ and $1\leq j\leq t$) so that $[x_{ij}]\in M_{ms,nt}$
  has $\|[x_{ij}]\|_2=1$.
\end{rk}

Below, we show that in many cases, $\hsnorm\cdot$ is comparable with
the operator norm for the right module maps~$T$ under
consideration. This allows us to estimate $\|T\|$ and $\|T\|_{cb}$,
and these estimates are used to find some upper bounds for
$\|T\|_{cb}/\|T\|$.  

\begin{prop}\label{prop:HS}
  Let~$m,n\in \bN\cup\{\infty\}$. If $T\colon M_{m, n}\to M_{m, n}$ is
  a right $D_n$-module map with column operators $\{a_j \colon 1\leq
  j<n+1\}$, then
  \[\hsnorm{T}=\sup_j\|a_j\|\leq \|T\|.\]
\end{prop}
\begin{proof}
  Recall that the column operators $a_j\in\L(\bC^m)$ of~$T$ were
  defined in Remark~\ref{rk:rep}.  Suppose, for convenience of
  notation, that $n<\infty$. Let~$a\in \L((\bC^m)^n)$ be the diagonal
  direct sum of $a_1,\dots,a_n$, so that
  $a(\xi_1,\dots,\xi_n)=(a_1\xi_1,\dots,a_n\xi_n)$ for
  $\xi_1,\dots,\xi_n\in \bC^m$. Then $\|a\|=\max_j\|a_j\|$. Since $T$
  is a right $D_n$-module map, we have
  \begin{align*}
    \hsnorm{T}&=\sup\{\|Tx\|_2\colon x\in M_{m,n},\ \|x\|_2\leq 1\}\\
    &= \sup\Big\{\sqrt{\sum_{j=1}^n\|T(x)e_j\|^2}\colon x\in M_{m,n},\ \|x\|_2\leq 1\Big\}\\
    &= \sup\Big\{\sqrt{\sum_{j=1}^n\|T(xe_je_j^*)e_j\|^2}\colon x\in M_{m,n},\ \|x\|_2\leq 1\Big\}\\
    &= \sup\Big\{\sqrt{\sum_{j=1}^n\|a_j(xe_j)\|^2}\colon x\in M_{m,n},\ \sum_{j=1}^n\|xe_j\|^2\leq 1\Big\}\\
    &= \sup\Big\{\sqrt{\sum_{j=1}^n\|a_j(\xi_j)\|^2}\colon \xi_j\in \bC^m,\ \sum_{j=1}^n\|\xi_j\|^2\leq 1\Big\}\\
    &= \sup\{\|a\xi\|\colon \xi\in (\bC^m)^n,\ \|\xi\|\leq 1\}=\|a\|=\max_j\|a_j\|.
  \end{align*}
  Moreover, if~$\eta\in \bC^m$ and $1\leq j\leq n$ then $\|\eta
  e_j^*\|=\|\eta\|$, so 
  \begin{align*}
    \|T\|&=\sup\{\|Tx\|\colon x\in M_{m,n},\ \|x\|\leq 1\}\\
     &\geq \sup\{\|T(\eta e_j^*)\| \colon \eta\in \bC^m,\ \|\eta\|\leq 1\}
     = \sup\{ \|a_j\eta\|\colon \eta\in\bC^m,\ \|\eta\|\leq 1\}=\|a_j\|,
  \end{align*}
  so $\|T\|\geq \max_j\|a_j\|$.

  If~$n=\infty$ then the proof is similar.
\end{proof}
\goodbreak

The following lemma will be used to obtain a useful inequality in the
other direction in Proposition~\ref{prop:cb} below.
\begin{lem}\label{lem:norms}
  Let~$m,n\in \bN\cup\{\infty\}$ with $k=\min\{m,n\}<\infty$. If
  $T\colon M_{m, n}\to M_{m, n}$ is a linear map, then
  $\|T\|\leq \sqrt k\hsnorm{T}$. 
\end{lem}
\begin{proof}
  Suppose $k=m\leq n$ and $\hsnorm{T}=1$. For $x\in M_{m, n}$,
  let~$\lambda_1\geq \lambda_2 \geq \dots \geq\lambda_k$ be the
  eigenvalues of~$xx^*$.  We have
  \[\|x\|^2=\lambda_1 \leq \|x\|_2^2=\sum_{j=1}^k\lambda_j \leq k
  \lambda_1=k\|x\|^2,\] so $\|x\| \leq \|x\|_2\leq \sqrt k\|x\|$. Hence
  $\|T(x)\| \leq \|T(x)\|_2 \leq\|x\|_2 \leq \sqrt k\|x\|$, so $\|T\|\leq
  \sqrt k$.

  If~$m>n$, consider the map $T^*\in \L(M_{n, m})$.
\end{proof}

If $c_1\in M_{m,n}$ and $k\in\bN$, then we write $c_1^{(k)}$ for the
block-diagonal operator in $M_k(M_{m,n})$ with $k$ copies of $c_1$
running down the diagonal. Similarly, if $c=
\begin{smallvect}
  c_1\svdots c_\ell
\end{smallvect}$ where $c_1,\dots,c_\ell\in M_{m,n}$, then $c^{(k)}=
\Bigg[\begin{smallmatrix}
  c_1^{(k)}\svdots c_\ell^{(k)}
\end{smallmatrix}\Bigg]$.
The utility of this notation is revealed by observing that if
$T=a\odot b$ and $s,t\in \bN$, then $T_{s,t}=a^{(s)}\odot b^{(t)}$.

\begin{prop}\label{prop:cb}
  Let\/ $\ell,n\in \bN$, let\/ $k=\min\{\ell,n\}$ and
  let\/ $K=\min\{\ell^2,n\}$.  If\/~$T\colon M_{\ell, n}\to M_{\ell,
    n}$ is a right $D_n$-module map, then \[\|T\|_{cb}=\|T_{k,1}\|
  \leq \sqrt K\hsnorm{T}.\] In particular,
  $\|T\|_{cb}=\|T_{n,1}\|\leq\sqrt n\hsnorm{T}$.
\end{prop}
\begin{proof}
  By Remark~\ref{rk:rep},~$T$ is an elementary operator, and there are
  matrices $a_1,\dots,a_n\in M_\ell$ and $b_1,\dots,b_n\in D_n$ such that
  $Tx=\sum_{j=1}^na_jxb_j$
  and \[\|T\|_{cb}=\tfrac12(\|a\|^2+\|b\|^2)\quad\text{where}\quad
  a=[a_1\ \dots\ a_n]\text{ and }b=
  \begin{bmatrix}
    b_1\\\vdots\\b_n
  \end{bmatrix}.\] By~\cite[Theorem~3.3]{timoney03} we have
  $\|T\|_{cb}=\|T_n\|$.

  By Lemma~\ref{lem:conv}, $\Wme(b)$ is convex.
  By~\cite[Proposition~2.4]{timoney03}, the set $\Wme((a^*)^{(\ell)})$ is convex,
  so intersects $\Wme(b^{(\ell)})$, so $\|T\|_{cb}=\|T_\ell\|$. Hence
  \[\|T_k\|=\min\{\|T_\ell\|,\|T_n\|\}=\|T\|_{cb}.\]

  By Lemma~\ref{lem:smith}, $\|T\|_{cb}=\|T_{k,1}\|$.  By
  Remark~\ref{rk:hsnorm}, the map $T_{k,1}\in\L_{D_n}(M_{k\ell, n})$
  satisfies $\hsnorm{T_{k,1}}=\hsnorm{T}$, hence $\|T_{k,1}\|\leq
  \sqrt K\hsnorm{T}$ by Lemma~\ref{lem:norms}.
\end{proof}

\section{Upper bounds for $C(m,n)$}\label{sec:upperbounds}

For $n,m\in\bN\cup\{\infty\}$, recall that
\[C(m,n)=\sup\{\|T\|_{cb} \colon T\in \L_{D_n}(M_{m, n}),\ \|T\|\leq
1\}.\] We have $C(m,1)=C(1,n)=C(m,2)=1$ by Remark~\ref{rk:one} and
Theorem~\ref{thm:2col}.

We will now give some upper bounds for $C(m,n)$.

\begin{prop}\label{prop:Cincr}
  If $m\leq m'$ and $n\leq n'$ then $C(m,n)\leq C(m',n')$. In other
  words, $C$ is an increasing function for the product order.
\end{prop}
\begin{proof}
  Given $T\in \L_{D_{n}}(M_{m, n})$ with $\|T\|=1$, let
  $T'\in\L_{D_{n'}}(M_{m', n'})$ be the map
  \[T'(x)=
  \begin{bmatrix}
    T(qxp)& 0_{m\times (n'-n)}\\0_{(m'-m)\times n}& 0_{(m'-m)\times (n'-n)}
  \end{bmatrix},\quad x\in M_{m',n'}\] where 
  \[q=[I_m\ 0_{m\times (m'-m)}]\in
  M_{m,m'}\quad\text{and}\quad p=
  \begin{bmatrix}
    I_n\\0_{(n'-n)\times n}
  \end{bmatrix}\in M_{n',n}.\] 
  It is easy to see that $\|T'\|=\|T\|=1$ and
  $\|T\|_{cb}=\|T'\|_{cb}$. Hence $C(m,n)\leq C(m',n')$.
\end{proof}

\begin{quesc}
  Do we have $C(m,n)C(s,t)\leq C(ms,nt)$ for all $m,n,s,t\ge1$? 
  A proof purporting to give a positive answer appeared
  in~\cite{lt-laa}, which we have since realised is
  incorrect~\cite{lt-corr}. 
\end{quesc}

\begin{lem}\label{lem:2colbound}
  Let $n\in\bN$ with $n\ge2$. If $y\in M_{m,n}$ and
  $\|y(e_ie_i^*+e_je_j^*)\|\leq 1$ for $1\leq i<j\leq n$, then
  $\|y\|\leq \sqrt{n/2}$.
\end{lem}
\begin{proof}
  Let $p_{ij}=e_ie_i^*+e_je_j^*$ for $1\leq i<j\leq
  n$.  Since each
  $p_{ij}$ is a projection, we have
  $\|yp_{ij}y^*\|=\|(yp_{ij})(yp_{ij})^*\|=\|yp_{ij}\|^2\leq
  1$. Moreover, \[\sum_{1\leq i<j\leq n} p_{ij}=(n-1)I_n.\] Hence
  \begin{align*} \|y\|=\sqrt{\|yy^*\|}=\sqrt{\frac1{n-1}\left\|\sum_{1\leq
        i<j\leq n}yp_{ij}y^*\right\|}
    &\leq\sqrt{\frac1{n-1}\sum_{1\leq i<j\leq n} \|yp_{ij}y^*\|}\\
    &\leq \sqrt{\frac1{n-1}\binom n2} =\sqrt{ \frac n2}.\qedhere
  \end{align*}
\end{proof}

The following simple estimate applies to arbitrary linear maps between
operator spaces, and is analogous to the well-known bound $\|T_n\|\leq
n\|T\|$ (\cite[Exercise~3.11]{paulsen}, due to~Smith).
\begin{lem}\label{lem:rootbound}
  Let $k,m,n\in\bN$. For any $T\in \L(M_{m,n})$, we have
  \[\|T_{k,1}\|\leq \sqrt k\|T\|.\]
\end{lem}
\begin{proof}
  There is $x \in M_{k ,1}(M_{m,n})$, say $x = \begin{bmatrix} x_1 \\ x_2 \\
    \vdots \\ x_k \end{bmatrix}$ (where $x_j \in M_{m,n}$ for $1 \leq
  j \leq k $), with $\|x\| = 1$ and $\|T_{k ,1}(x)\| =\|T_{k,1}\|$.
  Clearly we can write $T_{k , 1}(x) = \begin{bmatrix} Tx_1 \\ Tx_2 \\
    \vdots \\ Tx_k \end{bmatrix}$, and since $\|x_j\|\leq \|x\|\leq 1$
  for $1\le j\leq n$, we have
  \begin{align*}
    \|T_{k,1}\|&=\|T_{k,1}(x)\|\\&=\sqrt{\|T_{k,1}(x)^*T_{k,1}(x)\|}\\
    &= \sqrt{\left\|\sum_{j=1}^k T(x_j)^*T(x_j)\right\|}\\
    &\leq \sqrt{\sum_{j=1}^k \|T(x_j)^*T(x_j)\|}\\
    &=\sqrt{\sum_{j=1}^k \|T(x_j)\|^2}\\
    &\leq \|T\|\sqrt{\sum_{j=1}^k \|x_j\|^2}\\
    &\leq \sqrt k\|T\|.\qedhere
  \end{align*}
\end{proof}

\begin{thm}\label{thm:upperbound}
  If~$m,n\in\bN$ and $n\ge2$ then $C(m,n)\leq \sqrt{\min\{m,n/2\}}$.
\end{thm}
\begin{proof}
  Let $T\in\L_{D_n}(M_{m,n})$ with $\|T\|=1$ and $\|T\|_{cb}=C(m,n)$.
  By Proposition~\ref{prop:cb} and Lemma~\ref{lem:rootbound}, we
  have \[C(m,n)=\|T\|_{cb}=\|T_{m,1}\|\leq \sqrt m\|T\|=\sqrt
  m\] so it only remains to show that $C(m,n)\leq \sqrt{n/2}$.

  Since $\|T\|_{cb}=\|T_{n,1}\|$ by Proposition~\ref{prop:cb}, there
  is $x\in M_{m,1}(M_{m,n})$ with $\|x\|=1$ such that $y=T_{m,1}(x)$
  has $\|y\|=C(m,n)$. Let $a_1,\dots,a_n$ be the column operators
  of~$T$, so that
  \[ T=[a_1\ \dots\ a_n]\odot
  \begin{bmatrix}
    e_1e_1^*\\\vdots\\e_ne_n^*
  \end{bmatrix}.\] 
  Given $i,j$ with $1\leq i<j\leq n$,
  consider \[S=[a_i\ a_j]\odot
  \begin{bmatrix}
    e_1e_1^*\\ e_2e_2^*
  \end{bmatrix}\in
  \L_{D_2}(M_{m,2}).\] Clearly, $\|S\|\leq \|T\|$. If $w=[xe_i\
  xe_j]\in M_{m^2, 2}$, then $\|w\|\leq \|x\|=1$. Moreover,
  $y(e_ie_i^*+e_je_j^*)\in M_{m^2, n}$ can be recovered from
  $S_{m,1}(w)\in M_{m^2, 2}$ by padding with $n-2$ columns of zeros,
  so $\|S_{m,1}(w)\|=\|y(e_ie_i^*+e_je_j^*)\|$. Since
  $\|S\|=\|S\|_{cb}$ by Theorem~\ref{thm:2col}, we have
  \[ 1=\|T\|\geq \|S\|=\|S\|_{cb} \geq
  \|S_{m,1}(w)\| = \|y(e_ie_i^*+e_je_j^*)\|.\] By
  Lemma~\ref{lem:2colbound},
  \[ C(m,n)=\|y\|\leq \sqrt{n/2}.\qedhere\]
\end{proof}

Fix~$n\in\bN$. The sequence~$C(2,n),C(3,n),C(4,n),\dots$ is
increasing by Proposition~\ref{prop:Cincr}. We will now show that it
is eventually constant.

In~\cite[Theorem~1.3]{timoney07}, the second author establishes an
exact formula for the norm of an elementary operator $T$, which we now
recall. If $\ell\in\bN$ and $X,Y$ are positive semi-definite elements
of $M_\ell$, then the \emph{tracial geometric mean} of~$X$ and~$Y$ is
\[\tgm(X,Y)=\|\sqrt {X}\,\sqrt{Y}\|_1=\trace\sqrt{\sqrt X Y\sqrt
  X}\] where $\|\cdot\|_1$ denotes the trace-class norm
on~$M_\ell$. If~$T$ is an elementary operator on~$M_m$ which is
represented by $a\in M_{1,\ell}(M_m)$ and $b\in M_{\ell,1}(M_m)$, then
the formula is:
\begin{equation}
  \|T\|=\sup\{\tgm(X,Y)\colon X\in \Wm(a^*),\ Y\in \Wm(b)\}.
  \label{eq:tgm}\tag{$\dagger$}
\end{equation}
In fact, a generalisation of this formula is shown to hold for
elementary operators on any C*-algebra~$A$.

We need to show that~(\ref{eq:tgm}) holds in the rectangular case,
too. If~$T$ is an elementary operator on~$M_{m,n}$ with $n> m$ which
is represented by~$a\in M_{1,\ell}(M_m)$ and $b\in M_{\ell,1}(M_n)$,
consider the map \[\tilde T\colon M_n\to M_n,\qquad x\mapsto
\begin{bmatrix}
  T(px)\\0_{(n-m)\times n}
\end{bmatrix},
\] where $p\in \B(\bC^n,\bC^m)$ is the orthogonal
projection onto the linear span of 
$\{e_1,\dots,e_m\}$, which is viewed simultaneously 
as~$\bC^m$ 
and as a subspace 
of~$\bC^n$.  That is, $\tilde T$ is ``$T$ applied to the first 
$m$ rows, and zero on the remaining 
rows''. Clearly, $\|T\|=\|\tilde T\|$. If $a=[a_1\ \dots\ a_\ell]$ and
$\tilde a=[\tilde a_1\ \dots\ \tilde a_\ell]$ 
where $\tilde a_j=
\begin{smallbmatrix}
  a_j&0\\0&0
\end{smallbmatrix}\in M_n$ is ``$a_j$ padded with $n-m$ zero
rows and columns'', then $\tilde T$ is represented by $\tilde a,b$,
and 
$\Wm(\tilde a^*)=\{rX\colon r\in [0,1],\ X\in \Wm(a^*)\}$. 
So, since $\tgm(rX,Y)=\sqrt r\tgm(X,Y)$ for $r\in [0,1]$,
\[\|T\|=\|\tilde T\|=\sup\{\tgm(X,Y)\colon X\in \Wm(a^*),\ Y\in
\Wm(b)\}.\] If $T$ is an elementary operator on $M_{m,n}$ with $n<m$,
then~(\ref{eq:tgm}) still holds, as may be seen by considering $T^*$.

\begin{rk}
  The tracial geometric mean (or, sometimes, its square) is called
  \emph{fidelity} in quantum information theory~\cite{Q1,Q2,Q3,Q4},
  where it is interpreted as a measure of the closeness of two quantum
  states (positive semi-definite trace-class operators with
  trace~$1$).
\end{rk}

\begin{thm}\label{thm:tgmcols}
  If $1\leq n< m\leq \infty$, then $C(m,n)=C(n,n)$.
\end{thm}
\begin{proof}
  Suppose first that $m<\infty$. The supremum $C(m,n)$ is then
  attained, so there is $T\in \L_{D_n}(M_{m,n})$ with $\|T\|=1$ and
  $\|T\|_{cb}=C(m,n)$.  By Proposition~\ref{prop:cb},
  $\|T\|_{cb}=\|T_{n,1}\|$. By~(\ref{eq:tgm}), there are unit vectors
  $\xi_1,\dots,\xi_n$ in $\bC^{m}$ and $\eta\in\bC^n$, and
  $r_1,\dots,r_n\in[0,1]$ with $\sum_{j=1}^n r_j^2=1$ such that the
  vector $\xi=\begin{smallvect}r_1\xi_1\svdots
    r_n\xi_n\end{smallvect}$ satisfies
  \[\|T\|_{cb}=\tgm\big(Q((a^*)^{(n)},\xi),Q(b,\eta)\big).\] Let~$K$
  be an $n$-dimensional subspace of~$\bC^{m}$ containing
  $\xi_1,\dots,\xi_n$, and let us identify~$K$ with~$\bC^n$. Then
  writing $p$ for the orthogonal projection of~$\bC^m$ onto~$K$, let
  $\tilde a_j=p a_j|_K$, let $\tilde a=
  [\tilde a_1\ \dots\ \tilde a_n]$ and let $\tilde T$ be the elementary operator
  on~$M_n$ represented by $\tilde a,b$.  By our choice of~$K$, we have
  $\|\tilde T\|_{cb}=\|T\|_{cb}$ and $Q(\tilde a^*,\xi)=Q(a,\xi)$ for
  $\xi\in K$, so \begin{align*} \|\tilde T\|&=\sup\{ \tgm\big(Q(\tilde
    a^*,\xi),Q(b,\eta)\big) \colon \xi\in K,\ \eta\in\bC^n,
    \|\xi\|,\|\eta\|=1\} \\&\leq \sup\{
    \tgm\big(Q(a^*,\xi),Q(b,\eta)\big) \colon \xi\in \bC^m,\
    \eta\in\bC^n, \|\xi\|,\|\eta\|=1\}\\&=\|T\|.
  \end{align*}
  Since  $\tilde T$ is a right $D_n$-module map, we
  have \[C(n,n)\geq \frac{\|\tilde T\|_{cb}}{\|\tilde T\|}\geq
  \frac{\|T\|_{cb}}{\|T\|}=C(m,n)\geq C(n,n)\] and hence
  $C(n,n)=C(m,n)$.

  The case $m=\infty$ now follows by the argument of
  Corollary~\ref{cor:infty2}.
\end{proof}

\begin{rk}
  This reduces Theorem~\ref{thm:2col} to the $2\times 2$ case, but
  does not appear to greatly simplify the proof.
\end{rk}

\section{More than two columns}\label{sec:morethan2}

We now give some examples which establish non-trivial lower bounds
for~$C(m,n)$ when $n\ge3$.  The matrix extremal numerical range of an
$\ell$-tuple $[a_1\ \dots\ a_\ell]^*$ is closely connected to the
joint numerical range of the operators $a_ja_i^*$ for $1\leq i<j \leq
\ell$. Moreover, the joint numerical range of three matrices (even
three hermitian matrices) need not be convex, and an explicit example
of this phenomenon is given in~\cite[Example~1.1]{li-poon}. 

Let
\[
a_1=\begin{bmatrix}
  1&0\\0&1
\end{bmatrix},\quad
a_2=\begin{bmatrix}
  1&0\\0&-1
\end{bmatrix}\quad\text{and}\quad
a_3=\begin{bmatrix}
  0&1\\1&0
\end{bmatrix}.
\]
It is easy to see that the joint numerical range of the operators
$a_ja_i^*$ for $1\leq i<j\leq 3$ is affinely equivalent to a
$2$-sphere, so is not convex. Our first example is the map whose
column operators are $a_1,a_2,a_3$.
\begin{eg}\label{eg:2x3}
  The map  $T\colon M_{2, 3}\to M_{2, 3}$, \[T\colon
  \begin{bmatrix}
    a&c&e\\b&d&f
  \end{bmatrix}\mapsto
  \begin{bmatrix}
    a&c&f\\b&-d&e
  \end{bmatrix}
  \]
  is a right $D_3$-module map with \[\sqrt 2= \|T\| <
  \|T_{2,1}\|=\|T\|_{cb}=\sqrt{3}.\] So $C(2,3)=\sqrt{3/2}$.
\end{eg}
\begin{proof}
  $T$ is a right $D_3$-module map and a Hilbert-Schmidt isometry. By
  Lemma~\ref{lem:norms}, $\|T\| \leq \sqrt 2$, and we have equality
  since $\|T[
\begin{smallmatrix}
  1&0&0\\0&0&1
\end{smallmatrix}]\|=\sqrt2$.

By Proposition~\ref{prop:cb}, $\|T\|_{cb}\leq \sqrt 3$, and if
\[ x=\frac1{\sqrt 2}
\begin{bmatrix}
  1&{1}&0\\0&0&1\\0&0&1\\1&{-1}&0
\end{bmatrix}\quad \text{then}\quad T_{2,1}(x)=
\frac1{\sqrt 2}
\begin{bmatrix}
  1&1&1\\0&0&0\\0&0&0\\1&1&1
\end{bmatrix}
\] and $\|x\|=1$ while $\|T_{2,1}(x)\|=\sqrt 3$. So $\|T\|_{cb}=\sqrt
3=\|T_{2,1}\|$. Hence $C(2,3)\geq\sqrt{3/2}$, and we have equality by
Theorem~\ref{thm:upperbound}.
\end{proof}

Extending the previous example by one column yields:
\begin{eg}\label{eg:2x4}
  The map  $T\colon M_{2, 4}\to M_{2, 4}$, \[T\colon
  \begin{bmatrix}
    a&c&e&g\\b&d&f&h
  \end{bmatrix}\mapsto
  \begin{bmatrix}
    a&c&f&h\\b&-d&e&-g
  \end{bmatrix}
  \]
  is a right $D_4$-module map with \[\sqrt 2= \|T\| <
  \|T_{2,1}\|=\|T\|_{cb}=2.\] So $C(2,4)= \sqrt2$.
\end{eg}
\begin{proof}
  $T$ is a right $D_4$-module map and a Hilbert-Schmidt isometry. By
  Lemma~\ref{lem:norms}, $\|T\| \leq \sqrt 2$, and we have equality
  since $\|T[
\begin{smallmatrix}
  1&0&0&0\\0&0&1&0
\end{smallmatrix}]\|=\sqrt2$.
  By Proposition~\ref{prop:cb}, $\|T\|_{cb}\leq 2$, and if
  \[ x=\frac1{\sqrt 2}
  \begin{bmatrix}
    1&{1}&0&0\\0&0&1&1\\0&0&{1}&{-1}\\1&{-1}&0&0
  \end{bmatrix}\quad \text{then}\quad T_{2,1}(x)=
  \frac1{\sqrt 2}
  \begin{bmatrix}
    1&1&1&1\\0&0&0&0\\0&0&0&0\\1&1&1&1
  \end{bmatrix}
  \] and $\|x\|=1$ while $\|T_{2,1}(x)\|=2$. So
  $\|T\|_{cb}=2=\|T_{2,1}\|$.  Hence $C(2,4)\geq\sqrt2$, and
  Theorem~\ref{thm:upperbound} gives the reverse inequality.
\end{proof}
\goodbreak

Example~\ref{eg:2x4} may be generalised as follows:
\begin{thm}\label{thm:eg}
  For each $m\in\bN$ with $m>1$, there is a right $D_{m^2}$-module map
  $T\in \L(M_{m, m^2})$ with $\sqrt m=\|T\|<
  \|T_{m,1}\|=\|T\|_{cb}=m$. Hence $C(m,m^2)=\sqrt m$.
\end{thm}
\begin{proof}
  We have $C(m,m^2) \leq \sqrt m$ by Theorem~\ref{thm:upperbound}.
  
  Let~$\rho=e^{2\pi i/m}$ and consider the $m \times m$ matrices
  \[
  g= \begin{bmatrix}
    1 & 0 & \cdots & 0\\
    0 & \rho & \cdots & 0\\
    \vdots&  & \ddots \\
    0 & 0 & \cdots & \rho^{m-1}
  \end{bmatrix}, \qquad h = \begin{bmatrix}
    0 & 0 & \cdots &0& 1\\
    1 & 0 & \cdots &0& 0\\
    0 & 1 & \cdots &0& 0\\
    \vdots&  & \ddots \\
    0 & 0 & \cdots & 1 & 0
  \end{bmatrix}
  \]
  so that $h$ is the matrix for the $m$-cycle permutation $\alpha =
  (1\ 2\ \ldots\ m)$.

  For $1\leq j\leq m^2$, let $0\leq r<m$ and $1\leq s\leq m$ with
  $j=mr+s$, and define
  \[ a_j=g^{-(s-1)}h^{-r}.\] Take $T$ to be the right $D_{m^2}$-module
  map with column operators $a_j$ ($1 \leq j \leq m^2$).  Since
  each~$a_j$ is unitary, $T$ is a Hilbert-Schmidt isometry and so
  $\|T\| \leq \sqrt{m}$ by Lemma~\ref{lem:norms}.
  (By Proposition~\ref{prop:cb},
  $\|T\|_{cb} \leq m$, but we will not actually need that.)

  For~$1\leq i\leq m$ and~$1\leq j\leq m^2$, let~$v^i_j\in \bC^m$ be
  the vector
  \[v^i_j=\rho^{(i-1)(s-1)}e_{\alpha^r(i)}\quad\text{where $j=mr+s$
    with $0\leq r<m$, $1\leq s\leq m$}\] and define $x^i\in M_{m,m^2}$
  by
  \[ x^i=\sum_{\substack{j=mr+s\\ 0\leq r<m\\ 1\leq s\leq m}} v^i_j
  e_j^*\] Observe that $a_jv_j^i=e_i$ for every~$j$. Hence,
  $Tx^i=e_iw^*$ where~$w=\sum_{j=1}^{m^2}e_j\in \bC^{m^2}$, and so
  $\|Tx^i\|=\|e_i\|\,\|w\|=m$.

  If\[ x=
  \begin{bmatrix}
    x^1\\x^2\\\vdots\\x^m
  \end{bmatrix}\in M_{m,1}(M_{m,m^2})\quad\text{then}\quad T_{m,1}(x)=
  \begin{bmatrix}
    e_1w^*\\e_2w^*\\\vdots\\ e_mw^*
  \end{bmatrix}=
  \begin{bmatrix} 
    e_1\\e_2\\\vdots\\ e_m
  \end{bmatrix}w^*,
  \]
  so $\|T_{m,1}(x)\|=\|[e_1^*\ \dots\ e_m^*]\|\,\|w\|=m^{3/2}$.  On
  the other hand, a calculation shows that the rows of $x$ are
  mutually orthogonal and have norm~$\sqrt m$, and so $\|x\| =
  \sqrt{m}$. Hence $ m \leq \|T_{m,1}\|=\|T\|_{cb}$ and $C(m, m^2)
  \geq \|T\|_{cb}/\|T\| \geq m/\sqrt{m} = \sqrt{m}$.
\end{proof}

\begin{rk}
  By Lemma~\ref{lem:rootbound}, \[\|T_{m-1,1}\|\leq \sqrt
  {m-1}\|T\|<\|T_{m,1}\|=\|T\|_{cb}=\sqrt m\|T\|=m\hsnorm{T}\] for the
  $D_{m^2}$-module maps~$T$ on $M_{m,m^2}$ constructed in this
  proof. Thus the estimates of Proposition~\ref{prop:cb} and
  Theorem~\ref{thm:upperbound} are sharp, at least for $n=m^2$.
\end{rk}

\begin{cor}\label{cor:manycols} If $m,n\in\bN$ with $n\geq m^2$ then
  $C(m,n)=\sqrt m$.  
\end{cor}
\begin{proof}
  Since $C(m,n)$ is increasing in~$n$, this is an immediate
  consequence of Theorems~\ref{thm:upperbound} and~\ref{thm:eg}.
\end{proof}

\begin{thm}\label{thm:fracbound}
  If $m,n\in\bN$ with $2\leq n\leq m^2$, then $C(m,n)\geq\sqrt{n/m}$.
\end{thm}
\begin{proof}
  Let~$T\in \L_{D_{m^2}}(M_{m,m^2})$ and $x\in M_{m^2}$ be as in the
  proof of Theorem~\ref{thm:eg}. Let $a_1,\dots,a_{m^2}$ be the
  column operators of~$T$ and consider the map $S\in
  \L_{D_n}(M_{m,n})$ whose column operators are $a_1,\dots,a_n$. Also,
  let $x_j=xe_j\in \bC^{m^2}$ be the $j$th column of~$x$ and
  let~$y=[x_1\ \dots\ x_n]\in M_{m^2,n}$. By following the earlier
  argument, it is not hard to see that $S_{m,1}(y)$ is the matrix in
  $M_{m^2,n}$ whose columns are the first~$n$ columns of~$T_{m,1}(x)$,
  and hence that $\|S_{m,1}(y)\|=\sqrt n$. Since $\|y\|\leq \|x\|=1$,
  we have
  \[\|S\|_{cb}\geq \|S_{m,1}\|\geq \|S_{m,1}(y)\|=\sqrt n.\] Clearly
  $\|S\|\leq \|T\|\leq \sqrt m$. Hence \[ C(m,n) \geq
  \frac{\|S\|_{cb}}{\|S\|} \geq \sqrt\frac nm\,.\qedhere\]
\end{proof}
\begin{rk}
  For $(m,n)=(2,3)$, the operator~$S$ in this proof was considered in
  Example~\ref{eg:2x3}, and we have equality in the bounds
  $\|S\|_{cb}\geq\sqrt n$ and $\|S\|\leq\sqrt m$ in this
  case.
\end{rk}

We now summarise the best bounds we have obtained for $C(m,n)$. 
\begin{cor}
  Let $m,n\in \bN\cup\{\infty\}$.

  (i) If $m=1$ or $n\in \{1,2\}$ then $C(m,n)=1$.

  (ii) If $m\geq n$ then $C(m,n)=C(n,n)$.

  (iii) If $n\geq m^2$ then $C(m,n)=\sqrt m$.

  (iv) If $2\leq n\leq m^2$ then $\sqrt{\max\{\lfloor \sqrt n\rfloor,
      n/{\lceil \sqrt n\rceil}\}}\leq C(m,n)\leq \sqrt{\min\{m,n/2\}}$.
\end{cor}
\begin{proof}
  Statements (i)--(iii) and the upper bound in~(iv) have been
  discussed already, in Remark~\ref{rk:one}, Theorem~\ref{thm:2col},
  Corollary~\ref{cor:manycols} and Theorem~\ref{thm:upperbound}.
  
  Suppose that $2\leq n\leq m^2$. Let $k=\lfloor \sqrt n\rfloor$. Then
  $m\geq\sqrt n\geq k$ and $n\geq k^2$, and $C$ is increasing by
  Proposition~\ref{prop:Cincr}, so $C(m,n)\geq C(k,k^2)=\sqrt k$ by
  Theorem~\ref{thm:eg}.  Similarly, if we write $\ell=\lceil \sqrt
    n\rceil$ then $m\geq \ell$ so $C(m,n) \geq C(\ell,n)\geq
    \sqrt{n/\ell}$ by Theorem~\ref{thm:fracbound}.
\end{proof}

\begin{ques}
  If~$m\leq2$ or $n\leq4$ then these bounds yield exact values
  of~$C(m,n)$, but we have been unable to find the exact values of
  $C(m,n)$ in many other cases. In particular, what is $C(3,5)$?
\end{ques}

\begin{rk}
  It seems improbable that the lower bounds we have obtained could be
  sharp in general. In particular, it would seem surprising if
  $C(6,6)$ turned out to be no larger than $C(2,4)=\sqrt2$. 
\end{rk}

\begin{ques}
  Is $C(n,n)$ strictly increasing in~$n$ for $n\ge2$?
\end{ques}

We now answer Question~\ref{ques:JLTT} in the negative.  Recall that a
masa in~$\B(H)$ is said to be~\emph{discrete} if it is generated by
its minimal projections. 

\begin{cor}\label{cor:answer}
  If~$H$ is an infinite-dimensional Hilbert space and $D$ is a masa in
  $\B(H)$, then there is a bounded right $D$-module map $T\colon
  \K(H)\to \K(H)$ which is not completely bounded.
\end{cor}
\begin{proofc}
  First suppose that $H$ is separable and~$D$ is discrete.  By
  considering the minimal projections in~$D$, we may identify $H$ with
  $\bigoplus_{m\ge1} H_m$ where $H_m=\bC^{m^2}$ for $m\ge1$, in such a
  way that the minimal projections of~$D$ are identified with the
  coordinate projections of~$H_m$.

  Let~$p_m$ be the orthogonal projection
  in~$D\subseteq\B(H)$ with range~$H_m$. By Theorem~\ref{thm:eg},
  there is a right $D_{m^2}$-module map~$T_{(m)}\colon \B(H_m)\to
  \B(H_m)$ with $\|T_{(m)}\|=1$ and $\|T_{(m)}\|_{cb} \geq \sqrt
  m$. The map $T\colon \K(H)\to \K(H)$, $x\mapsto \sum_{m\ge1}^\oplus
  T_{(m)}(p_mxp_m)$ has $\|T\|=1$ and $\|T\|_{cb}\geq
  \sup_{m\ge1}\|T_{(m)}\|_{cb}=\infty$.

  This concludes the proof in the case that~$D$ is discrete. In
  particular, this shows that that there is a bounded right
  $\ell^\infty(\bN)$ module map
  \[T\colon \K(\ell^2(\bN))\to \K(\ell^2(\bN)),\quad Tx=\sum_{j\in\bN}
  a_j xE_{j,j}\] with $\|T\|_{cb}=\infty$, where
  $a_j\in\K(\ell^2(\bN))$ and $E_{j,j}$ denotes the $j$th diagonal
  matrix unit in~$\K(\ell^2(\bN))$. (That $Tx$ has this form for $x =
  \sum_{j=1}^n x E_{j,j}$ follows as in Remark~\ref{rk:rep}, and in
  general by taking limits.)  The argument of
  Proposition~\ref{prop:col} shows that $\|T_{1,n}\|=\|T\|$ for all
  $n\in\bN$, and it follows that for any cardinal~$d$, the map
  $T_{1,d}$ acting on $\K(\ell^2(\bN)\otimes \ell^2(d),\ell^2(\bN))$
  given by $T_{1,d}[x_s]_{0\leq s< d}= [Tx_s]_{0\leq s< d}$ has
  $\|T_{1,d}\|=\|T\|<\infty$ and $\|T_{1,d}\|_{cb}=\infty$.

  If~$D$ is a masa in $B(H)$ which
  has an infinite-dimensional
  discrete summand, then we can use the discrete case above. If $D$ has no
  such summand,
  there must be a cyclic subspace $\overline {D \xi}$ for $D$ and a
  corresponding infinite-dimensional projection $P \in D$ such that $PD$
  acting on $PH$
  is unitarily equivalent to
  $L^\infty(M, \mu)$ acting on $L^2(M, \mu)$ for some non-atomic finite perfect measure
  space $(M ,\mu)$
  (by~\cite[Lemma 1.2]{segal}).
  By Maharam's theorem~\cite{maharam}, passing to a summand we can assume that
  $M = [0,1]^p$ for some cardinal $p$ (and $\mu$ the product probability
  measure).
  Hence we reduce to
  ~$H=L^2([0,1]^p)$ and $D=L^\infty([0,1]^p)$.

  Choose a measurable partition $A_1, A_2,\ldots $ of $[0,1]$ into sets
  of positive measure.  Consider $H$ as the Hilbert space direct sum of
  $H_j:=L^2(A_j\times [0,1]^{p-1})$, and let $P_j \in B(H)$ be the
  orthogonal projection of $H$ onto $H_j$; note that $P_j$ is a
  multiplication operator, so $P_j\in D$.
  In each $H_j$ we choose an orthonormal basis $(\phi^j_s)_{0\leq s< d}$
  where~$d=\aleph_0^p$, and let $U \colon H \to \ell^2(\bN)
  \otimes\ell^2(d)$ be the unitary with $U(\phi^j_s) = e_j \otimes e_s$
  where the $e_j$ and $e_s$ are the standard basis vectors of
  $\ell^2(\bN)$ and $\ell^2(d)$, respectively. Let~$V\colon
  \ell^2(\bN)\to H$ be an isometry and consider the map
  \[
  S\colon \K(H)\to \K(H),\quad Sy=V
  T_{1,d}(V^*yU^*)U.
  \]
  Since $P_j = U^* (E_{j,j} \otimes I)U$
  and $y = \sum_{j\in\bN} yP_j$, we have
  $Sy=V\sum_{j\in\bN} a_jV^*yP_j$. In particular, $S$ is a right
  $D$ module map. Moreover, if $\alpha(y)= V^*yU^*$ then
  $\alpha_n$ maps the unit ball of $M_n(\K(H))$ onto the unit ball of
  $M_n(\K(\ell^2(\bN)\otimes \ell^2(d),\ell^2(\bN)))$, hence
  $\|S\|=\|T_{1,d}\|<\infty$ and $\|S\|_{cb}=\|T_{1,d}\|_{cb}=\infty$.
\end{proofc}

  

\begin{rk}
  Under the same hypotheses, using weakly convergent sums in place of
  the norm convergent sums in this construction provides a bounded
  right $D$-module map $\B(H)\to\B(H)$ which is not completely
  bounded.
\end{rk}

\begin{rkc}\label{rk:tensorpowers}
  If~$T\colon M_{m,n}\to M_{m,n}$ is a right $D_n$-module map with
  $1=\|T\|<\|T\|_{cb}$, then the $k$th tensor power of~$T$, that is,
  the map $T^{\otimes k}\colon M_{m^k,n^k}\to M_{m^k,n^k}$ is a right
  $D_{n^k}$-module map, and $\|T^{\otimes k}\|_{cb} = \|T\|_{cb}^k\to
  \infty$ as $k\to \infty$. In general, we may have $\|T^{\otimes
    k}\|>1$ (for example, if~$T=\mathrm{id}\oplus S$ where
  $\mathrm{id}$ is the identity map on $M_\ell$ and~$S$ is a map
  on~$M_\ell$ with norm one and completely bounded norm greater than
  one). However, if~$T\colon M_{2,3}\to M_{2,3}$, $T\begin{smallbmatrix}a&c&e\\b&d&f\end{smallbmatrix}=
  \tfrac1{\sqrt2}
  \begin{smallbmatrix}a&c&f\\b&-d&e\end{smallbmatrix}$ is the map of
  Example~\ref{eg:2x3} scaled to have norm~one, then
  $\hsnorm{T^{\otimes k}}=2^{-k/2}$ and so $\|T\|=\|T^{\otimes k}\|=1$
  by Lemma~\ref{lem:norms}.  Thus Example~\ref{eg:2x3} may be used in
  place of Theorem~\ref{thm:eg} to establish
  Corollary~\ref{cor:answer}.
\end{rkc}

\section{Subsets of the right module maps}
\label{sec:failures}

For $m,n\in\bN$, let~$\S(m,n)$ be a subset of $\L(M_{m,n})$ containing
a nonzero mapping and let
\[C_{\S}(m,n)=\sup\Big\{ \frac{\|T\|_{cb}}{\|T\|} \colon T\in
\S(m,n),\ T\ne0\Big\}.\] Above, we have considered the case
$\S(m,n)=\L_{D_n}(M_{m,n})$ and have shown that the corresponding
function $C=C_\S$ can take values larger than~$1$. On the other hand,
if $\S$ is the class of Schur multipliers, then $C_\S$ is
identically~$1$. It seems natural to ask for which classes of
operators~$\S$ we still have $C_\S(m,n)>1$ for some~$m,n$. Of course,
if~$\S(m,n)\subseteq\L_{D_n}(M_{m,n})$ then $1\leq C_\S(m,n)\leq
C(m,n)$.
\bigskip

Let $m,n\in\bN$. Given $\alpha\in S_m$, let $u_\alpha$ be the
corresponding permutation unitary satisfying
$u_\alpha(e_i)=e_{\alpha(i)}$ for $1\leq i\leq m$.  Let \[\P(m,n)=\Big\{
[u_{\alpha_1}\ \dots\ u_{\alpha_n}] \odot
\begin{smallvect}[\Bigg]
  e_1e_1^*\svdots e_ne_n^*
\end{smallvect}\colon
\alpha_j\in S_m\Big\}\subseteq \L_{D_n}(M_{m,n}).\] This is a natural
class of right $D_n$-module maps in which to seek maps with larger cb
norm than norm. Indeed, if we drop the right modularity requirement,
then the classic example of such a map is the transpose of a square
matrix, which is a carefully chosen permutation of the matrix
entries;~$\P$ is precisely the set of right $D_n$-module maps which
are permutations of the matrix entries. We initially looked for
examples in this class, and having had no luck, were eventually led to
Examples~\ref{eg:2x3} and~\ref{eg:2x4}, and so to
Theorem~\ref{thm:eg}. Since we concentrated on the $2\times n$ and the
$3\times 3$ cases, it is nice to be able to offer the following
explanation for this initial failure.

\begin{prop}
  $C_\P(2,n)=C_\P(3,3)=1$.
\end{prop}
\begin{proof}
  By Theorem~\ref{thm:2col}, $C_\P(2,n) \leq C(2,n)=1$, so
  $C_\P(2,n)=1$. Alternatively, since $S_2$ is abelian, this is an
  immediate consequence of \cite[Remark~2.5]{timoney03}.

  Now consider $u\odot e\in \P(3,3)$ where $u=[u_1\ u_2\ u_3]$ and
  $u_j=u_{\alpha_j}$ for some $\alpha_j\in S_3$. Observe that if $u_0$
  is a unitary matrix in~$M_m$ then the norms and completely bounded
  norms of $u\odot e$ and $u_0u\odot e$ coincide. So, taking
  $u_0=u_1^{-1}$, we may assume that $\alpha_1$ is the identity
  permutation. Similarly, conjugating each $\alpha_j$ by some
  $\alpha_0\in S_3$ will not change the norm or completely bounded
  norm of the corresponding elementary operator. Hence up to symmetry
  there are three cases to consider:
  \begin{enumerate}
  \item $\alpha_2=(1\ 2\ 3)$ and $\alpha_3=(1\ 3\ 2)=\alpha_2^{-1}$;
  \item $\alpha_2=(1\ 2)$ and $\alpha_3=(1\ 2\ 3)$; and
  \item $\alpha_2=(1\ 2)$ and $\alpha_3=(1\ 3)$.
  \end{enumerate}
  In the first case, the unitaries all commute and hence
  $\|T\|=\|T\|_{cb}$ by~\cite[Remark~2.5]{timoney03}.  In both of the
  latter two cases, \[U=\{u_ju_i^* \colon 1\leq i<j\leq 3\} = \{u_{(1\
    2)}, u_{ (1\ 2\ 3)}, u_{(1\ 3)}\}\] and the joint numerical range
  of these three unitaries contains zero, since for every $u\in U$ we
  have $\langle ue_1,e_1\rangle=0$. Hence $\Wme(\tfrac1{\sqrt3}u^*)$
  contains a positive semidefinite diagonal $3\times 3$ matrix of
  trace~1, and Lemma~\ref{lem:conv} shows that $\Wme(e)$ consists of
  all such matrices. Hence
  \[\Wme(\tfrac1{\sqrt3}u^*)\cap\Wme(e)\ne \emptyset\] and so $T=u\odot
  e$ has $\|T\|_{cb}=\|T\|$.
\end{proof}
\goodbreak

However, a more persistent search reveals that $C_\P$ is not constant.
\begin{eg}\label{eg:P32}
  If
  \[T=[u_{(1)}\ u_{(1\ 2)}\ u_{(1\ 3)}\ u_{(2\ 3)}]\odot \begin{smallvect}[\Bigg]
    e_1e_1^*\\e_2e_2^*\\e_3e_3^*\\e_4e_4^*
  \end{smallvect} \in \L_{D_4}(M_{3,4}),\] then $\|T\|=\sqrt{3}$ and
  $\|T_{2,1}\|> 1.0775\sqrt 3$. Hence
  \[C_\P(3,4)\geq \frac{\|T_{2,1}\|}{\|T\|}>1.0775.\]
\end{eg}
\begin{proof}
  We have $\|T\|\leq \sqrt 3\hsnorm{T}=\sqrt 3$ by
  Lemma~\ref{lem:norms}, and the lower bound is given by considering
  the norm one matrix $\Big[
  \begin{smallmatrix}
    1&0&0&0\\0&1&0&0\\0&0&1&0
  \end{smallmatrix}\Big]$.
  Let
  \[ x=
  \begin{bmatrix}
    -2/3 & 0 & 0& 2/3 \\
    0 & 1 & 0 & 0\\
    0 & 0 &  1/\sqrt{2} &0\\
    0 & 0 &  1/\sqrt{2} &0\\
    1/3 & 0 & 0& 2/3 \\
    -2/3 & 0 & 0& -1/3 
  \end{bmatrix}.\] Observe that $\|x\|=1$, since we can reorder the
  rows and columns to recognise it as the direct sum of two $3\times
  2$ matrices with orthonormal columns. Now \[
  T_{2,1}(x) =
  \begin{bmatrix}
    -2/3 & 1 & 1/\sqrt{2} & 2/3 \\
    0 & 0 & 0 & 0\\
    0 & 0 & 0 & 0\\
    0 & 0 & 0 & 0\\
    1/3 & 0  & 0 & -1/3\\
    -2/3 & 0 & 1/\sqrt{2} & 2/3 
  \end{bmatrix},
  \]
  and a computation with Mathematica reveals that $\|T_{2,1}(x)\|^2$
  is the largest root of $18x^3-72x^2+33x-2=0$, and hence that
  $\|T_{2,1}(x)\|>1.0775\sqrt 3$.
\end{proof}

\begin{rk}
  Numerical estimates obtained from a GNU~Octave program using the
  tracial geometric mean formula~(\ref{eq:tgm}) give an improved lower
  bound for $\|T_{2,1}\|$ for the operator~$T$ in the preceding
  example
  of $1.13\|T\|$.
\end{rk}

\begin{cor}
  $C_\P(\infty,\infty)=\infty$.
\end{cor}
\begin{proofc}
  Let~$T$ be the map of Example~\ref{eg:P32}. Considering the tensor
  powers $T^{\otimes k}$, we see that $T^{\otimes k}\in \P(3^k,4^k)$. Since $\hsnorm{T^{\otimes k}}=1$, we have $\|T^{\otimes k}\|=3^{k/2}=\|T\|^k$, by Lemma~\ref{lem:norms}. Hence
   \[C_\P(\infty,\infty)\geq \sup_{k\ge1} \frac{\|T^{\otimes
    k}\|_{cb}}{\|T^{\otimes k}\|} = \sup_{k\ge1} \left(\frac {\|T\|_{cb}}{\|T\|}\right)^k=
  \infty.\qedhere\]
\end{proofc}

\begin{ques}
  If~$\min\{m,n\}<\infty$, is it ever true that $1< C_\P(m,n)=C(m,n)$?
\end{ques}

Finally, we pose a question about the class of module maps whose
column operators are unitary:
\[\U(m,n)=\Big\{
[u_1\ \dots u_n] \odot
\begin{smallvect}
  e_1e_1^*\svdots e_ne_n^*
\end{smallvect}\colon u_j\in \U(M_m),\ 1\leq j\leq n\Big\}\subseteq
\L_{D_n}(M_{m,n})\] where $\U(M_m)$ is the set of unitary operators
in~$M_m$.  The examples constructed in Theorem~\ref{thm:eg} are
in~$\U(m,m^2)$, so~$C(m,m^2)=C_\U(m,m^2)$ for all $m\ge1$.

\begin{ques}
  Is $C(m,n)=C_\U(m,n)$ for all $m,n\ge1$? 
\end{ques}

\end{document}